\documentclass[12pt]{amsart}
\usepackage{amssymb,latexsym}
\usepackage{epsf}

\newcommand{\F}{\mathbb{F}}
\newcommand{\Q}{\mathbb{Q}}
\newcommand{\Z}{\mathbb{Z}}
\newcommand{\QQ}{\mathcal{Q}}

\newcommand{\C}{\mathcal{C}}
\newcommand{\E}{\mathcal{E}}

\newcommand{\J}{{\mathcal J}}
\newcommand{\K}{{\mathcal K}}

\newcommand{\cZ}{\mathcal{Z}}
\newcommand{\ovb}{\overline{B}}
\newcommand{\X}{\bf{X}}

\def\FAA{\Q \kern .05em \langle y_1,y_2,\dots \rangle}
\def\NC{\Z \kern .05em \langle y_1,y_2,\dots \rangle}

\theoremstyle{plain}

\newtheorem{lemma}{Lemma}[section]
\newtheorem{theorem}{Theorem}[section]
\newtheorem{proposition}{Proposition}[section]
\newtheorem{definition}{Definition}[section]
\newtheorem{conjecture}{Conjecture}[section]

\theoremstyle{remark}

\newtheorem{example}{Example}[section]

\newcommand{\vanish}[1]{}

\begin{document}
\title[Properties of $\Sigma _{D_n}$ and  $\Sigma (D_n,p)$]
{Properties of the descent algebras of type $D$}
\author{Stephanie van Willigenburg}
\address{Department of Mathematics, University of British Columbia, 1984 Mathematics Road, Vancouver, BC, V6T 1Z2, Canada}
\email{steph@math.ubc.ca}
\thanks{Partially supported by the National Sciences and Engineering Research Council of Canada.}
\begin{abstract}
We establish simple combinatorial descriptions of the radical and irreducible 
representations specifically for the 
descent algebra of a Coxeter group of type $D$ over any field. 

\end{abstract}

\maketitle
\section{Introduction \label{intro}}

In enumerative combinatorics, quasisymmetric functions arise naturally in a 
variety of of areas, for 
example, the chromatic symmetric function of a graph \cite{chow}, the enumerator 
of partitions of a poset 
\cite{gessel, stembridge}, and the flag vector of a graded poset 
\cite{ehrenborg}. Moreover, the 
universality of the Hopf algebra of quasisymmetric  functions $\QQ $ has 
recently been studied by Aguiar 
and Bergeron (personal communication). The dual of this algebra, formed from the 
descent algebras of the 
symmetric groups, also arises in a variety of contexts such as Hochschild 
homology 
\cite{bergeron-hochschild}, card shuffling and hyperplane arrangements 
\cite{fulman}, and planar binary 
trees \cite{loday-ronco}, in addition to being isomorphic to the Hopf algebra of 
noncommutative symmetric 
functions \cite{gelfand-etal}.
Since a descent algebra exists for every Coxeter group \cite{solomon-mackey} it 
is natural to study those 
descent algebras stemming from other Coxeter groups (\emph{e.g.} 
\cite{bergeron-bergeron, bergberg-ht, 
bergeron-vW}) in the hope that a deeper understanding of $\QQ $ or the descent 
algebras of the symmetric 
groups might be obtained.   

In this paper we extend the study initiated in \cite{bergeron-vW} to establish 
straightforward combinatorial
descriptions of some algebraic properties of the descent algebras of (the 
Coxeter groups of) type $D$. 
 Throughout we utilise notation first seen in \cite{bergeron-vW} which allows us 
to state our results in a vein similar to that in \cite{atkinson-vW, 
garsia-reutenauer} for the symmetric groups case and \cite{bergeron-bergeron, 
bergeron-n, vanwilli-hyper} for the hyperoctahedral groups case. In these cases 
this notation 
not only allowed results to be stated clearly and concisely but also aided the 
realisation of the complete algebraic structure in each case. Definitions are 
given in the remainder of this section, and we 
state some results concerning our algebras over a field of characteristic zero in 
Section ~\ref{char0}. We 
then generalise this to a field of finite characteristic in Section 
~\ref{charp}. Finally we discuss open 
problems that could be addressed next.

The author is grateful to Louis Billera and the referees for valuable suggestions on earlier 
drafts of the paper.

\subsection{The descent algebras of Coxeter groups}\label{dacox}

Let $W$ be a Coxeter group with generating set $S$, let $J$ be a subset of  $S$,  
$W_J$ be the subgroup generated by $J$,  
$X_J$ ($X_J^{-1}$) be the unique set of minimal
length left (right) coset representatives of $W_J$,
and ${\mathcal X}_J$ be the formal sum of the elements
in $X_J$.   Solomon proved \cite[Theorem 1]{solomon-mackey} for $J,K,L$ being 
subsets of $S$

\begin{equation}{\mathcal X}_J{\mathcal X}_K=\sum _L a_{JKL}{\mathcal 
X}_L\label{sol}\end{equation}
where $a_{JKL}$ is the number of elements $x\in X_J^{-1}\cap X_K$ such that
$x^{-1}Jx\cap K=L$. Moreover the ${\mathcal X}_J$ form a basis for the 
\emph{descent algebra of $W$} (over 
$\Q$), denoted $\Sigma _W$, and the radical of $\Sigma _W$ is spanned by all 
differences 
$\mathcal{X}_{J}-\mathcal{X}_{K}$
where $J$ and $K$ are conjugate subsets of $S$ \cite[Theorem 3]{solomon-mackey}.

If $J$ is a subset of $S$, let $\phi _J$ be the permutation character of $W$ 
acting on the right cosets of 
$W_J$, and $c_J$ be a Coxeter element of $J$, that is, a product of all the 
elements of $J$ taken in some 
order. If we choose a set of representatives of the conjugacy classes of subsets of $S$ we 
find (for example 
\cite{mdapmod}) the columns of the matrix $R= [ \phi _J (c_K)]$ with rows and 
columns indexed by 
this set can be taken as the irreducible representations of $\Sigma _W$.

We delay the discussion of descent algebras over a field of prime characteristic 
until Section 
~\ref{charp}.

\subsection{The Coxeter groups of type $D$}

From now on we consider $D_n$, the $n$-th Coxeter group of type $D$,  which for 
our purposes will be the 
group
acting on the set $$\{-n, \ldots ,-1,1,\ldots ,n\}$$ whose
Coxeter generators are the set
$S=\{s_{1'},s_1,s_2,\ldots ,s_{n-1}\}$, where 
$s_i$ is the product of transpositions $(-i\!
-\! 1,\, -i)(i,i+1)$ for
$i=1,2,\ldots ,n-1$, and
$s_{1'}$ is the product of transpositions $(-2,1)(-1,2)$. Observe from this 
definition that if $\sigma \in 
D_n$ then $\sigma(-i) = - \sigma (i)$ and the parity of $\sigma$ is even, 
\emph{i.e.}, the multiplicity of 
negative numbers in $ \{ \sigma(1), \ldots, \sigma(n)\}$ is even.

\section{The descent algebras of type $D$}\label{char0}

It transpires that for the Coxeter group $D _n$,   (\ref{sol}) can be rewritten 
in a manner that is 
open to combinatorial manipulation  \cite[Theorem 1]{bergeron-vW}. This in turn 
allows us to derive 
algebraic properties of the algebra in a manner similar to the  symmetric or 
hyperoctahedral groups cases 
(see \cite{garsia-reutenauer, bergeron-bergeron} resp.). However, before we can 
do this we need to relabel 
our basis elements.

Recall a composition $\kappa$ of a non-negative integer $n$ is an ordered list  
$[\kappa_1,\kappa_2,\ldots ,\kappa_k]$ of positive integers whose sum is $n$, 
denoted by
$\kappa\vDash n$. We   call the integers $\kappa_1,\kappa_2,\ldots ,\kappa_k$ 
the
 \textit{components} of
$\kappa$. By convention  $\kappa = [\ ]$ is the unique composition of $0$. If 
$\kappa$ is a non-increasing 
list, \emph{e.g.}, $[3,2,1,1]$,  we call $\kappa$ a \emph{partition}, and the 
$\kappa _i$ are called 
\emph{parts}.

 There exists a natural correspondence between the subsets of $S$ and the 
multiset
$\C(n)$, consisting of the union of the sets \(\C _{<n}=\{\kappa|\kappa\vDash 
m,m\leq
        n-2\}\), \( \C_1=\{\kappa|\kappa\vDash n,\kappa_1=1 \}\) 
\( \C_n=\{\kappa|\kappa\vDash n,\kappa_1\geq 2 \}\) and 
  \( \C_n'=\{\kappa|\kappa\vDash n,\kappa_1\geq
2\}\). Observe   $\C_n$ and
$\C_n'$ are two  copies of the same set. The     subset   
corresponding to $\kappa\in\C(n)$ is
 \begin{enumerate}
\item \(\{s_{\kappa_0}, s_{\kappa_0+\kappa_1},\ldots , s_{\kappa_0+\ldots
+\kappa_{(k-1)}}\}\) if \(\kappa\in\)$\C_{<n}$ where \(\kappa_0=n-m,\)
\item \(\{s_{1'}, s_{1}, s_{1+\kappa_2},\ldots , s_{1+\kappa_2+\ldots 
+\kappa_{(k-1)}}\}\)  if 
\(\kappa\in\)\( \C_1\),
\item \(\{s_{1'},s_{ \kappa_1},\ldots ,  s_{\kappa_1+\ldots
+\kappa_{(k-1)}}\}\)  if \(\kappa\in\)\( \C_n\), 
\item \(\{s_{1}, s_{\kappa_1},\ldots , s_{\kappa_1+\ldots +\kappa_{(k-1)}}\}\) 
if \(\kappa\in\)\( \C_n'\).
 
\end{enumerate}

Our relabelling is then simply 
$B_\kappa :={\mathcal X}_{ J^c}$ where  $J^c$ is the complement of $J$ in $S$, 
and $\kappa$ is the 
composition in $\C(n)$ corresponding to
$J$ by the
 above bijection.

\begin{definition}
If $\kappa ,\nu\in\C(n)$ then we say that $\kappa\approx\nu$ if 
the components of $\kappa$ can be re-ordered to give to components
of
$\nu$, but $\kappa$ and $\nu$ do not satisfy either
\begin{enumerate}
\item $\kappa\in \C _n$, $\nu\in\C _n'$, all components are even, or
\item $\kappa\in \C _n'$, $\nu\in\C _n$, all components are even.
\end{enumerate}
\end{definition}

\begin{example} In $\C(6)$ let $[2,1,2,1], [4,2]\in\C _6$ and $[1,2,2,1]^\vee, 
[2,4]^\vee \in\C _6'$ where 
the $^\vee$ is simply to distinguish the compositions in $\C _6$ from those in 
$\C _6'$. Then 
$[2,1,2,1]\approx [1,2,2,1]^\vee$ but $[4,2]\not\approx 
[2,4]^\vee$.\end{example}

In fact this equivalence relation completely determines the conjugacy classes of 
subsets of $S$.

\begin{lemma}\label{conj}
$J$ and $K$ are conjugate subsets in $S$ if and only if $\kappa\approx\nu$ in 
$\C(n)$, where  $\kappa, \nu$  corresponds to $J^c,K^c$ respectively.
\end{lemma}

\begin{proof}Let $J, K\subseteq S$. If $J$ and $K$ are conjugate subsets then  
it follows that there exists $\sigma\in D_n$ 
such that for all $s\in J$ we have $s ^\sigma =\hat{s}$ for some $\hat{s}\in K$. 
 Since 
$\sigma$ is a bijection it follows that if $s=(-b,-a)(a,b), t=(-c,-b)(b,c)\in J$ 
then if $s^\sigma = (-e,-d)(d,e)\in K$ then $t^\sigma$ must be of the form 
$(-f,-e)(e,f)\in K$. Hence the compositions $\kappa$ and $\nu$ that correspond 
to $J^c$ and $K^c$ 
via the natural correspondence have the same components.

However, if all the components of $\kappa$ and $\nu$ are even, $\kappa\in \C_n$ 
and $\nu\in \C _n '$ then 
no such $\sigma \in D_n$ exists, for $\sigma$ would be a bijection which maps 
some $\{ s_i, s_{i+1}, \ldots, s_{i+2k}\}\subseteq J$ where 
$s_{i-1},s_{i+2k+1}\not\in J$, onto $\{ s_1', s_2, \ldots, 
s_{2k+1}\}\subseteq K$. From here, it is straightforward to deduce that for this 
to occur, the parity of $\sigma$ must be odd, and so does not belong to $D_n$.

Similarly if $\kappa\in \C_n '$ and $\nu\in \C _n $ then   $\sigma$ would be a 
bijection which maps  $\{ s_1', s_2, \ldots, s_{2m+1}\} \subseteq J$ where 
$s_{2m+2}\not\in J$, onto some $\{ 
s_i, s_{i+1}, \ldots, s_{i+2m}\}\subseteq K$. The parity of $\sigma$ must 
again be odd.

Since this argument is reversible it follows that $J$ and $K$ are conjugate if 
and only if 
$\kappa\approx \nu$ and we are done. \end{proof}

We are now ready to reinterpret the form of the radical.

\begin{theorem} The radical of $\Sigma _{D_n}$ is spanned by all $B_\kappa 
- B _\nu$ such that $\kappa \approx \nu$ where  $\kappa,\nu\in\C(n)$.
\end{theorem}

\begin{proof} This follows immediately from the description of the radical for 
$\Sigma _W$ \cite[Theorem 3]{solomon-mackey} and Lemma ~\ref{conj}.\end{proof}

Observe how our choice of notation gives a result whose statement has a similar 
flavour to the analogous result in the symmetric groups case \cite[Theorem 
1.1]{garsia-reutenauer} or the hyperoctahedral groups case \cite[Corollary 
2.13]{bergeron-n}.

The columns of the matrix $R= [ \phi _J (c_K)]$ can be taken as the 
irreducible representations of 
$\Sigma _W$, where the indexing set of rows and columns of $R$ are  pairwise 
non-conjugate subsets of $S$. Hence 
by Lemma ~\ref{conj} it   follows that for the descent algebra $\Sigma 
_{D_n}$ subsets whose complements correspond to  partitions in $\C(n) \setminus  \C_n '$ or  partitions in $\C_n '$ with all parts even form a suitable 
indexing set.

\section{The $p$-modular descent algebras of type $D$}\label{charp}

Since the structure constants of $\Sigma _W$ are the integers $a_{JKL}$, it 
follows the $\Z$-module  
spanned by all ${\mathcal X}_K$ forms a subring $\cZ _W$ of $\Sigma _W$,
 and for any prime $p$, all  combinations of the $\mathcal{X} _J$ that are integer multiples of $p$
form an an ideal, 
$\mathcal{P} _W$, of this subring. We define $\Sigma(W,p):=\cZ _W /\mathcal{P} 
_W$ to be the 
\textit{$p$-modular 
descent algebra} of $W$ and  $\Sigma(W,p)$ is our desired descent algebra over 
$\F _p$, the field of characteristic $p$.

In addition if ${\mathcal X}_J$ is a basis element of
 $\Sigma _W$, and $\rho $ is the natural homomorphism between ${\mathcal Z}_W$ 
and
 $\Sigma(W,p)$, then the set of all $\overline{{\mathcal X}}_J:=\rho ({\mathcal 
X}_J)$ forms a basis for 
$\Sigma(W,p)$. Observe
that the multiplicative action in $\Sigma(W,p)$ is the same as that in 
${\mathcal Z
}_{W}$ except we now reduce all coefficients modulo $p$. It was shown in 
\cite[Theorem 3]{mdapmod} that 
the radical of $\Sigma(W,p)$ is spanned by all 
 $\overline{{\mathcal  X}}_J-\overline{{\mathcal  X}}_K$ where $J,K$ are 
conjugate subsets 
 of $S$, together with all $\overline{{\mathcal  X}}_J$ for which $p$ divides 
 $a_{JJJ}$ -- the coefficient of of ${\mathcal X} _J$ in 
$\mathcal{X}_J\mathcal{X}_J\in \Sigma _W$. In 
addition it was shown a full set of distinct columns of the matrix $R= [ \phi _J 
(c_K)]$ whose entries have been reduced 
modulo $p$ can be taken as the irreducible representations of $\Sigma (W, p)$.

Following the conventions laid down so far in a natural way, let the basis of 
$\Sigma (D_n, p)$ be denoted 
by $\{ \ovb _\kappa\} _{\kappa\in\C(n)}$. Then the radical and irreducible 
modules can be described using 
the following two notions.

We say a composition has a \emph{component of multiplicity $m$} if $m$ of its 
components are of equal 
value, \emph{e.g.} $[1,3,1]$ has a component of multiplicity $2$,
and an \emph{$m$-regular partition} of a positive integer $n$ is a partition with 
no part divisible by $m$.

\begin{theorem}\label{dn-pmod}
Let $\kappa, \nu\in \C(n)$. If $p\neq 2$, the radical of $\Sigma (D_n,p)$ is 
spanned by all $\ovb _\kappa 
- \ovb _\nu$ such that $\kappa \approx \nu$, together with all $\ovb _\kappa$ 
where $\kappa$ contains a 
component of multiplicity $p$ or more. However if $p=2$ and $n$ is even the 
spanning set consists of all 
$\ovb _\kappa$, where $\kappa\neq [\ ]$;  if $p=2$ and $n$ is odd the 
spanning set consists of 
all $\ovb _\kappa$, where $\kappa\not\in \{ [\ ], [n], [n]^\vee\}, $ and $\ovb 
_{[n]} - \ovb 
_{[n]^\vee}$.\end{theorem}

\begin{proof}This result is a straightforward consequence of \cite[Theorem 3, 
Lemma 10]{mdapmod} and Lemma ~\ref{conj}.\end{proof} 

Again observe the similarity of Theorem ~\ref{dn-pmod} to \cite[Theorem 
2]{atkinson-vW} and \cite[Theorem 2]{vanwilli-hyper}.

A full set of distinct columns of the matrix $R= [ \phi _J (c_K)]$ whose entries 
are reduced modulo $p$ can be 
taken as the irreducible representations of $\Sigma (W, p)$,   and in fact it 
was shown in  \cite{mdapmod} that a suitable set consisted of  all columns indexed by a 
subset $J$ 
such that $p$ does not divide $a_{JJJ}$. Therefore considering the descent 
algebra $\Sigma (D_n, p)$, by \cite[Lemma 10]{mdapmod}, 
if $p\neq 2$  then subsets whose complements correspond to  $p$-regular partitions 
in 
$\C(n) \setminus \C_n '$ or   
$p$-regular partitions in $\C_n '$ with all parts even form a suitable indexing  set. If $p=2$ and $n$ is 
even then the subset whose complement corresponds to  $[\ ]$ is a suitable 
indexing set, whereas 
if $p=2$ and $n$ is odd, then
subsets whose complements correspond  to $[\ ]$ or  $[n]$ suffice.

\section{Conclusion}

With these results, the next avenue of research  would be to establish further 
properties for $\Sigma 
_{D_n}$ or $\Sigma(D_n,p)$, for example  a set of minimal idempotents, or the 
Cartan matrix. For 
this to be achieved the action of $\Sigma _{D_n}$ on ``$D_n$-Lie monomials"  
would have to be found,  as was necessary in the symmetric   and hyperoctahedral 
groups situations 
\cite{garsia-reutenauer, bergeron-n}. However there   has yet to be 
success in finding such monomials.

We  suspect, by observation of the symmetric and hyperoctahedral groups cases, 
that the action may be derived from multiplication in $\Sigma _{D_n}$. More 
precisely let $A=\{ a_1,a_2,\ldots, a_M\}$ be an alphabet   of $M$ distinct 
letters. The \emph{free Lie algebra} $L(A)$ is the algebra generated by the 
letters of $A$ and the bracket operation 
$$[P,Q]=PQ-QP.$$
Now consider homogeneous elements $P_i ^{(\kappa _i)}$ of degree $\kappa _i$ in 
$L(A)$. We call the concatenation product of $P_1 ^{(\kappa _1)},\ldots ,P_k 
^{(\kappa _k)}$ an \emph{$S_n$-Lie monomial} if $\sum _i \kappa _i = n$. 

Multiplication in $\Sigma _{S_n}$, with basis denoted by $\{ 
B_\kappa\}_{\kappa\vDash n}$ can be described as follows (say \cite[Theorem 
2]{vanwilli-sn}). If $\kappa, \nu$ are compositions of $n$, then
 \[B_{\kappa}B_{\nu}=\sum _{\boldsymbol{z}} B_{r(\boldsymbol{z})}\]
where the sum is over all matrices $\boldsymbol{z}=(z_{ij})$ with non-negative 
integer entries that satisfy
\begin{enumerate}
\item $\sum _i z_{ij}=\kappa _j$,
\item $\sum _j z_{ij}=\nu _i$.\end{enumerate}
For each matrix, $\boldsymbol{z}$,    $r(\boldsymbol{z})$ is the composition 
obtained by reading the non-zero   entries of the matrix $\boldsymbol{z}$ by 
row.
\begin{example}
In $\Sigma _{S_4}$ if calculating $B_{[2,1,1]}B_{[2,2]}$ then the matrices 
satisfying our conditions are
\[
\left(\begin{array}{ccc}
2&0&0\\
0&1&1
\end{array}\right)\quad
\left(\begin{array}{ccc}
0&1&1\\
2&0&0
\end{array}\right)\quad
\left(\begin{array}{ccc}
1&1&0\\
1&0&1
\end{array}\right)\quad
\left(\begin{array}{ccc}
1&0&1\\
1&1&0
\end{array}\right)
.\]

Hence \[B_{[2,1,1]}B_{[2,2]}=   B_{[2,1,1 ]} +   B_{[1,1,2 ]} + 2 B_{[1,1,1,1 
]}.\]\end{example}

By comparison of Theorem 2.1 \cite{garsia-reutenauer} and the above description 
of multiplication in 
$\Sigma _{S_n}$ using matrices we observe that the action of $\Sigma _{S_n}$ on 
$S_n$-Lie monomials can be stated as follows.

\begin{proposition}
Let $\kappa ,\nu\vDash n$ and $P^{(\kappa)}:=P_1 ^{(\kappa _1)}\ldots P_k 
^{(\kappa 
_k)}$ then
$$P^{(\kappa)}B_\nu=\sum _{\boldsymbol{z}} P^{(r(\boldsymbol{z}))}$$ where the 
sum 
is over all matrices $\boldsymbol{z}=(z_{ij})$ with non-negative integer entries 
that satisfy
\begin{enumerate}
\item $\sum _i z_{ij}=\kappa _j$,
\item $\sum _j z_{ij}=\nu _i$,
\item there is exactly one non-zero entry in each column.\end{enumerate}
For each matrix, $\boldsymbol{z}$,    $r(\boldsymbol{z})$ is the composition 
obtained by reading the non-zero   entries of the matrix $\boldsymbol{z}$ by 
row. Moreover $P^{(\kappa)}B_\nu=0$ unless adjacent components of $\kappa$ sum 
to give the components of $\nu$.  
\end{proposition}

An analogous proposition holds if we extend our definition of an $S_n$-Lie 
monomial to 
form a $B_n$-Lie monomial \cite{bergeron-n} and compare the action of $\Sigma 
_{B_n}$  on these \cite[Lemma 1.4]{bergeron-n} with multiplication in $\Sigma 
_{B_n}$ \cite[Theorem 1]{bergeron-bergeron}. Therefore since other results in 
this paper can be stated in a manner similar to analogous results in the 
symmetric and hyperoctahedral groups cases it is hoped that when $D_n$-Lie 
monomials $P^{(\kappa)}$ can be defined they will have the following properties stemming from 
multiplication in $\Sigma _{D_n}$ \cite[Theorem 1]{bergeron-vW}.

\begin{conjecture}
Let $\kappa ,\nu\in \C(n)$ and $P^{(\kappa)}:=P_1 ^{(\kappa _1)},\ldots ,P_k 
^{(\kappa 
_k)} $ then
$$P^{(\kappa)}B_\nu=\sum _{\boldsymbol{z}} \tilde{P}^{(r(\boldsymbol{z}))}$$ 
where the  
sum is over filled templates $\boldsymbol{z} \in Z(\kappa ,\nu)$ \cite[p 
701]{bergeron-vW} with exactly one non-zero entry in each 
column. 
Also $P^{(\eta)}$ is a summand of $\tilde{P}^{(r(\boldsymbol{z}))}$ if and only 
if $B_{(\eta)}$ is a summand of $\tilde{B}_{(r(\boldsymbol{z}))}$ as determined 
by \cite[Theorem 1]{bergeron-vW}.

Moreover $P^{(\kappa)}B_\nu=0$ unless adjacent components of $\kappa$ sum to 
give the components of $\nu$ and it is not the case that
\begin{enumerate}
\item $\kappa\in \C _n$, $\nu\in\C _n'$, all components are even, or
\item $\kappa\in \C _n'$, $\nu\in\C _n$, all components are even.
\end{enumerate}
\end{conjecture}


\begin{thebibliography}{99}
\bibitem{mdapmod}
 M. Atkinson G. Pfeiffer and  S. van Willigenburg, The $p$-modular descent 
algebras, Algebr. Represent. Theory 5 (2002) 101-113.
\bibitem{atkinson-vW}
M. Atkinson and S. van Willigenburg, The $p$-modular descent algebra of the 
symmetric group,  {Bull. London Math. Soc.} 29 (1997) 407-414. 
 \bibitem{bergeron-bergeron}
 F. Bergeron and N. Bergeron, A decomposition of the descent algebra of the 
hyperoctahedral group I,
  {J.  Algebra} 148 (1992) 86-97.
\bibitem{bergeron-n}
N. Bergeron, A decomposition of the descent algebra of the hyperoctahedral group 
II,  {J. Algebra} 148 (1992) 98-122. 
\bibitem{bergberg-ht}
 F. Bergeron N. Bergeron R. Howlett and D.
Taylor, A decomposition of the descent algebra
 of a finite Coxeter group,
 {J. Algebraic Combin.}
1   (1992) 23-44.
\bibitem{bergeron-vW}
N. Bergeron and  S. van Willigenburg,
A multiplication rule for the descent algebra of type $D$,  
 {J. Algebra} 206   (1998) 699-705. 
 \bibitem{bergeron-hochschild}
N. Bergeron and H. Wolfgang, 
 The decomposition of Hochschild cohomology and Gerstenhaber operations, 
 { 
J. Pure Appl. Algebra} 104  (1995) 243-265. 
\bibitem{chow}
T. Chow, Descents, quasi-symmetric functions, Robinson-Schensted for posets, and 
the chromatic symmetric
function,  {J. Algebraic Combin.} 10   (1999) 227-240. 
\bibitem{ehrenborg}
R. Ehrenborg, On posets and Hopf algebras,  {Adv. Math.} 119 (1996) 1-25. 
\bibitem{fulman}
J. Fulman, Descent algebras, hyperplane arrangements, and shuffling cards, 
 {Proc. Amer. Math. Soc.} 
129 (2001) 965-973.
\bibitem{garsia-reutenauer}
A. Garsia and C. Reutenauer, A decomposition of Solomon's
 descent algebra,  {Adv.  Math.} 77 (1989) 189-262.
\bibitem{gelfand-etal}
I. Gelfand D. Krob  A. Lascoux B. Leclerc V. Retakh and J-Y. Thibon, 
Noncommutative
symmetric functions,  {Adv. Math.} 112 (1995) 218-348. 
\bibitem{gessel}
I. Gessel, Multipartite $P$-partitions and inner products of skew Schur 
functions,  {Contemp. Math.}  
34 (1984) 289-317.
\bibitem{loday-ronco}
J-L. Loday and M. Ronco,
Hopf algebra of the planar binary trees,  
 {Adv. Math.} 139 (1998) 293-309. 
\bibitem{solomon-mackey}
 L. Solomon, A Mackey formula in the group ring of a Coxeter group,  {J.
 Algebra} 41 (1976) 255-268.
\bibitem{stembridge}
J. Stembridge, Enriched $P$-partitions,  
 {Trans. Amer. Math. Soc.} 349 (1997) 763-788.  
\bibitem{vanwilli-sn}
S. van Willigenburg, A proof of Solomon's rule,  {J. Algebra} 206 (1998) 693-698.
\bibitem{vanwilli-hyper}
S. van Willigenburg, The $p$-modular descent algebras of the hyperoctahedral 
group 
and dihedral group, St Andrews University Technical Report CS/96/1, 1996.
 \end{thebibliography}
\end{document}